\documentclass[leqno, 10.5pt]{amsart}

\usepackage{amsfonts,amsmath,amssymb,amsthm,amscd,enumerate,hyperref,comment}
\usepackage[dvipsnames]{xcolor}

\newtheorem{thm}{Theorem}[section]

\newtheorem{lem}{Lemma}[section]
\newtheorem{cor}{Corollary}[section]

\newtheorem{prop}{Proposition}[section]

\theoremstyle{remark}
\newtheorem{rmk}{Remark}[section]

\numberwithin{equation}{section}

\theoremstyle{definition}

\newcommand{\C}{\ensuremath{\mathbb{C}}}

\newcommand{\R}{\ensuremath{\mathbb{R}}}

\newcommand{\Z}{\ensuremath{\mathbb{Z}}}
\newcommand{\CP}{\ensuremath{\mathbb{CP}}}

\newcommand{\rS}{\ensuremath{\mathbb{S}}}

\newcommand{\na}{\nabla}

\newcommand{\pa}{\partial}
\newcommand{\vphi}{\varphi}

\newcommand{\tr}{\text{tr}}

\newcommand{\KN}{\mathbin{\bigcirc\mspace{-15mu}\wedge\mspace{3mu}}}

\begin{document}

\title[4D gradient Ricci solitons with (half) WPIC] {Four-dimensional gradient Ricci solitons with (half) nonnegative isotropic curvature}
\author[Huai-Dong Cao And Junming Xie]{Huai-Dong Cao And Junming Xie}

% Current address %%%%%%%%%%%%%%%%%%%%%%%%%%%%%%%%%%%%%%%%%%%%%%%%%%%%%%%%%%%%%%%%%%%%%%%
\address{Department of Mathematics, Lehigh University, Bethlehem, PA 18015}
\email{huc2@lehigh.edu}

\address{Department of Mathematics, Rutgers University, Piscataway, NJ 08854}
\email{junming.xie@rutgers.edu}

% Abstract %%%%%%%%%%%%%%%%%%%%%%%%%%%%%%%%%%%%%%%%%%%%%%%%%%%%%%%%%%%%%%
\begin{abstract}
	This is a sequel to our paper \cite{Cao-Xie}, in which we investigated the geometry of 4-dimensional gradient shrinking Ricci solitons with half positive (nonnegative) isotropic curvature. In this paper, we mainly focus on 4-dimensional gradient steady Ricci solitons with nonnegative isotropic curvature (WPIC) or half nonnegative isotropic curvature (half WPIC). In particular, for $4$D complete {\it ancient solutions} with WPIC, we are able to prove the 2-nonnegativity of the Ricci curvature and bound the curvature tensor $Rm$ by $|Rm| \leq R$. For 4D gradient steady solitons with WPIC, we obtain a classification result. We also give a partial classification of 4D gradient steady Ricci solitons with half WPIC. Moreover, we obtain a preliminary classification result for 4D complete gradient {\it expanding Ricci solitons} with WPIC. Finally, motivated by the recent work \cite{Li-Zhang:22}, we improve our earlier results in \cite{Cao-Xie} on 4D gradient {\it shrinking Ricci solitons} with half PIC or half WPIC, and also provide a characterization of complete gradient K\"ahler-Ricci shrinkers in complex dimension two among 4-dimensional gradient Ricci shrinkers.
\end{abstract}
\maketitle

% new section %%%%%%%%%%%%%%%%%%%%%%%%%%%%%%%%%%%%%%%%%%%%%%%%%%%%%%%%%%%%%%%%%%%
\section{Introduction}

A complete Riemannian manifold $(M^n, g)$ is called a {\em gradient Ricci soliton} if there exists a smooth (potential) function $f$ on $M^n$ such that the Ricci tensor $Rc$ of the metric $g$ satisfies the equation
\begin{equation} \label{eq:Riccisoliton}
Rc+ \na^2f = \rho g,
\end{equation} 
for some constant $\rho \in R$. 
Here $\na^2 f$ denotes the Hessian of $f$. The Ricci soliton is shrinking, or steady, or expanding if $\rho >0$, or $\rho =0$, or $\rho <0$, respectively.
Ricci solitons, introduced by Hamilton \cite{Ha:88} in the late 1980s (see also \cite {Ha:93}), are self-similar solutions to Hamilton's Ricci flow and a natural extension of Einstein manifolds.  They often arise as singularity models in the Ricci flow, thus it is important to classify them or understand their geometry and topology. 

In this paper, following our very recent work \cite{Cao-Xie}, we shall mainly focus on 4-dimensional gradient steady Ricci solitons with nonnegative, or half nonnegative, isotropic curvature. As is well-known, compact steady (and expanding)  Ricci solitons must be Einstein. In dimension $n=2$, Hamilton \cite{Ha:88} discovered the first example of a complete noncompact gradient steady soliton $\Sigma^2$ on $\mathbb R^2$, called the {\it cigar soliton}, where the metric is given explicitly by
$$ ds^2=\frac{dx^2 +dy^2}{1+x^2+y^2},$$
with potential function $f=-\log (1+x^2+y^2)$.  
The cigar soliton $\Sigma^2$ has positive curvature and is asymptotic to a cylinder of finite
circumference at infinity.  Furthermore, Hamilton \cite{Ha:88, Ha:93} showed that the cigar soliton is the only
complete steady soliton on a 2-dimensional manifold with
bounded (scalar) curvature $R$ which assumes its maximum
at the origin. For $n\geq 3$,  in the late 1980s Bryant \cite{Bryant}  proved  that
there exists, up to scalings, a unique complete rotationally symmetric gradient steady Ricci
soliton on $\Bbb R^n$; see Chow et al \cite{Chow-etal1}
for a detailed description. The Bryant soliton has positive sectional curvature, linear curvature decay,
and volume growth on the order of $r^{(n+1)/2}$. Here, $r$ denotes the geodesic distance from the origin. 
On the other hand,  in the K\"ahler case,
the first author \cite{Cao:96} constructed a complete $U(m)$-invariant gradient steady K\"ahler-Ricci soliton on $\mathbb{C}^m$, for $m\geq 2$, with positive sectional curvature, linear curvature decay, and volume growth on the order of $r^m$.  Recently, for $n\ge 3$, Lai \cite{Lai1} has found a family of $\Z_2 \times O(n-1)$-invariant $n$-dimensional gradient steady solitons on $\R^n$ with positive curvature operator. More recently, Apostolov-Cifarelli \cite{AC:23} and Chan-Conlon-Lai \cite{CCL:24} have constructed new examples of gradient steady K\"ahler-Ricci soliton on $\mathbb{C}^m$, including a one-parameter family of toric gradient steady K\"ahler-Ricci solitons on $\mathbb{C}^2$ with positive sectional curvature. For other constructions or additional examples of gradient steady solitons, see, e.g., \cite{Ivey:94,  DW:11, Yang:12, BDGW:15, Apple:17, BH:17, CD:20, Butts:21, Schafer:23, Wink:21, Wink:23} and the references therein. We also like to point out that Munteanu
and Wang \cite{MW:11} proved any $n$-dimensional complete noncompact
gradient steady Ricci soliton either has one end or splits as the product of the real line $\R$ with a compact Ricci-flat manifold.

In dimension $n=3$, B.-L. Chen \cite{ChenBL:09} showed that complete ancient solutions to the Ricci flow, in particular 3D shrinking and steady solitons, must have nonnegative curvature $Rm\geq 0$. By Hamilton's strong maximum principle \cite{Ha:86}, it follows that a 3D complete gradient steady soliton either has positive curvature $Rm>0$, or is flat, or is isometric to the product $\Sigma\times \R$ of the cigar soliton with the real line. 
Concerning the classification of positively curved gradient steady solitons, Brendle \cite {Brendle:13} confirmed a conjecture of Perelman \cite{Perelman:03} that the Bryant soliton is the
only 3-dimensional complete noncompact $\kappa$-noncollapsed gradient steady
soliton with positive sectional curvature; see also the related work by Deng and Zhu \cite{DZ19, DZ20}. 

On the other hand, Hamilton had conjectured the existence of a one-parameter family of 3-dimensional complete collapsed gradient steady solitons with positive sectional curvature, called flying wing in the sense that each is asymptotic to a sector with angle $\alpha \in (0, \pi)$, that connects the Bryant soliton ($\alpha=0$) and the cigar $\times \R$ ($\alpha=\pi$); see, e.g., \cite{CHe18}. 
Very recently, in a series of three papers, Lai \cite{Lai1, Lai2, Lai3} has proved the existence of flying wings conjectured by Hamilton. In particular,  Lai \cite{Lai2} showed that every 3D complete noncompact collapsed positively curved gradient steady soliton admits an $S^1$-symmetry, as conjectured by Hamilton and the first author. Moreover, for each $\alpha \in (0, \pi)$, Lai \cite{Lai3} showed that there exists a flying wing which is asymptotic to a sector of angle $\alpha$. 

For $n\ge 4$, based on a result of Zhang \cite{Zhang2:09}, Cao-Chen \cite{CC09} and Catino-Mantegazza \cite{CM11} proved independently that any complete {\it locally conformally flat} gradient steady soliton is either flat or isometric to the Bryant soliton up to scalings.  Furthermore, in dimension $n=4$,  Chen and Wang \cite{CW} showed that half-conformal flatness implies locally conformal flatness; thus, by \cite{CC09, CM11}, they are  either flat or isometric to the Bryant soliton up to scalings. 
Kim \cite{Kim:17} showed that any $4$-dimensional complete noncompact gradient steady soliton with harmonic Weyl curvature is either Ricci-flat or isometric to the Bryant soliton; see also very recent higher dimensional extensions \cite{Li:21, Kim:23}. 
Also, Brendle \cite {Brendle:14} extended his work \cite {Brendle:13} to dimension $n\ge 4$ under the extra assumption that the steady soliton is asymptotically cylindrical; see also related recent works by Deng and Zhu \cite{DZ3:20, DZ20}. For some of the recent progress on 4D gradient steady solitons, see, e.g., \cite {BCDMZ:22, MMS:23} and the references therein. 

 Let $(M^4, g, f)$ be a $4$-dimensional gradient Ricci soliton. Throughout the paper, we shall assume $M^4$ is oriented. Then, the space of 2-forms $\wedge^2(M)$ admits the orthogonal decomposition  $$\wedge^2(M) = \wedge^{+}(M) \oplus \wedge^{-}(M) $$
into the eigenspaces of the Hodge star operator $\star : \wedge^2(M) \to \wedge^2(M)$  of eigenvalues $\pm 1$. Smooth sections of $\wedge^+ (M)$ and $\wedge^-(M)$ are called {\it self-dual} and {\it anti-self-dual} 2-forms, respectively.  Accordingly, the Riemann curvature operator $ Rm : \wedge^2(M) \to \wedge^2(M)$ admits a block decomposition into  four pieces, 
\begin{equation} \label{eq:Rmdecomp}
	Rm = 
	\begin{pmatrix}
		A & B \\
		{}^{t}B & C
	\end{pmatrix}
	=
	\begin{pmatrix}
		W^+ + \frac{R}{12}I & \mathring{Rc} \\
		\mathring{Rc} & W^- +\frac{R}{12}I
	\end{pmatrix},
\end{equation}
where $W^{\pm}$ denote the self-dual and anti-self-dual part of the Weyl tensor, $R$ the scalar curvature, and $\mathring{Rc}$ the traceless Ricci part.  It turns out that $(M^4, g)$ has {\it positive isotropic curvature} (PIC), introduced by Micallef and Moore \cite{Micallef-M:88}, if  the $3\times 3$ matrices
\[ A := W^{+} + \frac{R} {12}I \qquad \mbox{and} \qquad  C:=W^{-} +\frac{R} {12}I \] are both 2-positive definite. Similarly, $(M^4, g)$ has {\it nonnegative isotropic curvature}, or {\it weakly positive isotropic curvature} (WPIC), if A and C are both 2-nonnegative (i.e., weakly 2-positive). Finally, if either $A$ or  $C$ is 2-positive (or 2-nonnegative), then we call $(M^4, g)$ half PIC (or half WPIC). 

Hamilton \cite{Ha:97}  showed that the PIC (or half PIC) condition is preserved by the Ricci flow in dimension four\footnote{Later, this was proved in all dimensions $n\geq 5$ by Brendle-Schoen \cite{Brendle-S:08} and Nguyen \cite{Nguyen:10} independently, and this property played an essential role in Brendle-Schoen's proof of the long standing $1/4$-pinching differentiable sphere theorem \cite{Brendle-S:08}.} and successfully initiated the investigation of 4D  Ricci flow with surgery under the PIC assumption; see also the work of Chen-Zhu \cite{Chen-Zhu:06}. Subsequently,  by using the Ricci flow with surgery developed in \cite{Ha:97, Chen-Zhu:06}, Chen-Tang-Zhu \cite{CTZ:12}  classified compact $4$-manifolds with PIC up to diffeomorphisms.  

Meanwhile, $4$-dimensional gradient shrinking Ricci solitons with PIC or WPIC have been classified recently by Li, Ni and Wang \cite{Li-Ni-Wang:18}.  In our very recent work \cite{Cao-Xie}, we studied 4-dimensional gradient shrinking Ricci solitons with half PIC or half WPIC.  In this paper, we shall explore the geometry and classifications of 4-dimensional complete gradient steady (and expanding) Ricci solitons with WPIC or half WPIC.  

Our first result concerns the geometry of 4-dimensional complete ancient solutions with WPIC. 

\begin{thm} \label{thm:4Dancient}
 Let $(M^4, g(t)), -\infty <t\leq 0,$ be a 4-dimensional complete ancient solution to the
Ricci flow with nonnegative isotropic curvature. Then, $g(t)$ has 2-nonnegative Ricci curvature and  satisfies the curvature estimate $$|Rm| \leq R.$$
Furthermore, if $g(t)$ has positive isotropic curvature then it has 2-positive Ricci curvature.
\end{thm}

\begin{cor}
Every $4$-dimensional complete gradient steady soliton $(M^4, g, f)$ with WPIC has 2-nonnegative Ricci curvature and bounded curvature 
\[|Rm| \leq R \leq 1,\] where $g$ is normalized so that $R+|\na f|^2=1$. 
\end{cor}

\begin{rmk}
Curvature estimates in the form of $|Rm|\leq CR$ for 4D gradient Ricci solitons were first derived by Munteanu-Wang \cite{MW} in the shrinking case (with bounded scalar curvature) and subsequently obtained for $4$D gradient steady solitons in \cite{CCui20}, \cite{Chan1}, \cite[Theorem 5.2]{Cao:22}, \cite{CFSZ} and \cite{CMZ:21} under certain conditions; see also related results in \cite{CRZ} for 4D Ricci shrinkers and \cite{CaoLiu} for 4D Ricci expanders.  
\end{rmk}

Our next  two main results are about classifications of  4-dimensional gradient steady Ricci solitons with WPIC and half WPIC, respectively.
 
\begin{thm} \label{thm:WPIC}
	Let $(M^4, g, f)$ be a $4$-dimensional complete noncompact, non-flat, gradient steady Ricci soliton with nonnegative isotropic curvature. Then, either
	\begin{itemize}
		\item[(i)] $(M^4, g, f)$ has PIC and 2-positive Ricci curvature, or
		\smallskip	
		\item[(ii)] $(M^4, g, f)$ is a gradient steady K\"ahler-Ricci soliton with positive holomorphic bisectional curvature\footnote{In fact, $(M^4, g, f)$ has $Rm> 0$ when restricted to the real $(1, 1)$-forms $\wedge^{1,1}_{\R}(M)\subset \wedge^2(M)$.}, or
		\smallskip
		\item[(iii)] $(M^4, g, f)$  is isometric to a quotient of either $N^3\times \R$, or $\Sigma^2\times \Sigma^2$, or $\Sigma^2\times \R^2$. Here, $\Sigma^2$ is the cigar soliton; $N^3$ is the Bryant soliton or a flying wing.
	\end{itemize}
\end{thm}

Since the cigar soliton and 3D flying wings are collapsed, and also there is no $\kappa$-noncollapsed gradient steady K\"ahler-Ricci soliton with positive holomorphic bisectional curvature by a result of  Deng-Zhu \cite{DZ2:20}, we immediately have the following 

\begin{cor} \label{cor steadyWPIC} Let $(M^4, g, f)$ be a $4$-dimensional complete noncompact,  non-flat, $\kappa$-noncollapsed gradient steady Ricci soliton with nonnegative isotropic curvature. Then, either 
	\begin{itemize}
		\item[(a)] $(M^4, g, f)$ has PIC and 2-positive Ricci curvature, or
				
		\smallskip	
		\item[(b)] $(M^4, g, f)$ is isometric to a quotient of the product $N^3\times \R$ of the 3D Bryant soliton with the real line.  
	\end{itemize}
\end{cor}

\begin{rmk} 
Based partly on the works of Hamilton \cite{Ha:86, Ha:97} and Chen-Zhu \cite{Chen-Zhu:06}, Brendle  \cite {Brendle:14} showed that any $4$D complete, $\kappa$-{noncollapsed} gradient steady Ricci soliton with {\bf uniformly} PIC (see also p.18 for the definition) is isometric to the Bryant soliton. However, the rigidity problem is still open if the strictly stronger assumption of uniformly PIC is replaced by PIC. 
\end{rmk}

\begin{thm} \label{thm:halfWPIC}
	Let $(M^4, g, f)$ be a 4-dimensional complete, non-flat, gradient steady Ricci soliton with half nonnegative isotropic curvature. Then, either
	\begin{itemize}
		\item[(i)] $(M^4, g, f)$ has half positive isotropic curvature, or
\smallskip
              \item[(ii)] $(M^4, g, f)$ is Calabi-Yau, hence locally hyperK\"ahler, or
\smallskip
              \item[(iii)] $(M^4, g, f)$ is Ricci-flat and self-dual, or                        
\smallskip		
		\item[(iv)] $(M^4, g, f)$ is a locally irreducible, non-Ricci-flat, gradient steady K\"ahler-Ricci soliton, or
	\smallskip	
		\item[(v)] $(M^4, g, f)$ is isometric to a quotient of either $N^3\times \R$, or $\Sigma^2\times \Sigma^2$, or $\Sigma^2\times \R^2$.
 Here $\Sigma^2$ is the cigar soliton, and $N^3$ is either the Bryant soliton or a flying wing.
	\end{itemize}
\end{thm}

\begin{rmk} Theorem \ref{thm:WPIC} and Theorem \ref{thm:halfWPIC} are valid under some slightly weaker assumption than WPIC (or half WPIC); see Remark  \ref{rmk:improving}.  
\end{rmk}

Furthermore, by using a Bony type strong maximum principle for degenerate elliptic equations from Brendle-Schoen \cite{Brendle-S:08}, we also obtain the following result for 4-dimensional complete gradient expanding Ricci solitons with WPIC.

\begin{thm} \label{thm:WPICexpander}
	Let $(M^4, g, f)$ be a 4-dimensional complete noncompact, non-flat, gradient expanding Ricci soliton with nonnegative isotropic curvature. Then, either
	\begin{itemize}
		\item[(i)] $(M^4, g, f)$ has positive isotropic curvature, or
		\smallskip		
		\item[(ii)] $(M^4, g, f)$ is a locally irreducible gradient expanding K\"ahler-Ricci soliton (with WPIC), or
		\smallskip
		\item[(iii)] $(M^4, g, f)$ is isometric to a quotient of either ${\hat N}^3\times \R$, where ${\hat N}^3$ is a 3-dimensional expanding Ricci soliton with $Rc> 0$, or ${\hat \Sigma} \times \R^2$, or ${\hat \Sigma}_1^2\times {\hat \Sigma}_2^2$, where ${\hat \Sigma}$, ${\hat \Sigma}_1^2$, and ${\hat \Sigma}_2^2$ are one of the 2-dimensional complete gradient expanding Ricci solitons in \cite{Bernstein-M:15} and \cite{Ramos:18} (see Theorem \ref{thm:expander} for more details).
	\end{itemize}
\end{thm}

Finally, motivated by the recent work of Li and Zhang \cite{Li-Zhang:22}, we observe that our earlier results in \cite{Cao-Xie} for 4-dimensional complete gradient shrinking Ricci solitons can be further improved.  In particular, we are able to provide a characterization of complete gradient K\"ahler-Ricci shrinkers, in complex dimension two, in terms of the nonnegativity of the middle eigenvalue $A_2$ of the matrix $A$ in (\ref{eq:Rmdecomp}).  

\begin{thm} \label{thm: compact}
	Let $(M^4, g, f)$ be a 4-dimensional complete gradient shrinking Ricci soliton with $H^1 (M, \Z_2)=0$, and $A_1\leq A_2 \leq A_3$ be the eigenvalues of $A= W^{+} + \frac{R} {12}I $.
	Then, $(M^4, g, f )$ is a K\"ahler-Ricci shrinker if and only if $A_2\ge 0$ (but $A_2\not >0$).
\end{thm}

The paper is organized as follows. In Section 2, we recall some basic facts which will be used in the proofs of Theorems \ref{thm:4Dancient}--\ref{thm: compact}. In Section 3, we study 4D gradient steady Ricci solitons with WPIC or half WPIC and prove Theorem \ref{thm:4Dancient}--Theorem \ref{thm:halfWPIC}. In particular, we show that any 4-dimensional {\it ancient solution} to the Ricci flow with WPIC must have 2-nonnegative Ricci curvature (Theorem \ref{prop:4Dancient}(a)) which is essential for deriving the curvature bound $|Rm|\le R$. In Section 4, we treat the expanding case and prove Theorem \ref{thm:WPICexpander}. In Section 5, we examine the shrinking case again and strengthen the results in our previous paper \cite{Cao-Xie} from which Theorem \ref{thm: compact} follows. Finally, in Section 6, we prove that the property of 2-nonnegative Ricci curvature is preserved by 4D Ricci flow with WPIC and briefly discuss PIC vs. uniformly PIC for 4D ancient solutions.

% new section %%%%%%%%%%%%%%%%%%%%%%%%%%%%%%%%%%%%%%%%%%%%%%%%%%%%%%%%%%%%
\section{Preliminaries}
In this section, we fix the notation and collect several basic facts and known results that we shall use later. We also refer the reader to our previous paper \cite{Cao-Xie} for more details.

% new subsection %%%%%%%%%%%%%%%%%%%%%%%%%%%%%%%%%%%%%%%%%%%%%%%%%%%%%%%%%%
\subsection{Nonnegative isotropic curvature}

Recall that a general Riemannian manifold $(M^n,g)$ of dimension $n\ge 4$ is said to have {\it nonnegative isotropic curvature}, or  {\it weakly positive isotropic curvature} (WPIC), if the Riemann curvature tensor $Rm=\{R_{ijkl}\}$ satisfies
the inequality 
\begin{equation} \label{eq:WPIC} 
	R_{1313}+R_{1414}+R_{2323}+R_{2424}-2R_{1234}\geq 0
\end{equation}
for any orthonormal four-frame $\{e_1,e_2,e_3,e_4\}$. For Riemannian manifolds with WPIC, we have the following basic properties.

\begin{lem} \textup{\bf (Brendle \cite{Brendle:10b} \& Micallef-Wang \cite{MW:93})} \label{lem:WPIC}
	Let $(M^n,g)$, $n\geq 4$, be a Riemannian manifold with nonnegative isotropic curvature. Then, for any orthonormal tangent frame $\{e_1, \cdots e_n\}$ and mutually distinct $i,j,k,l \in \{1,...,n\}$, we have

	\begin{itemize}
		\item[(a)] $ R_{ikik}+R_{ilil}+R_{jkjk}+R_{jljl} \geq 0; \label{eq1:WPIC} $
	\smallskip	
		\item[(b)] $ R_{ii}+R_{jj}-2R_{ijij}+(n-4)(R_{ikik}+R_{jkjk}) \geq 0; \label{eq2:WPIC} $
		
\smallskip
		\item[(c)] $ R_{ii}+R_{jj}-2R_{ijij} \geq 0;  \label{eq3:WPIC} $
\smallskip		
		\item[(d)] $ (n-4)R_{ii}+R \geq 0.  \label{eq4:WPIC} $
	\end{itemize}
	\end{lem}
For the reader's convenience, we also include a proof here. 
\begin{proof}
	Since $(M^n,g)$ has WPIC, inequality (\ref{eq:WPIC}) holds for any orthonormal four-frame $\{e_1,e_2,e_3,e_4\}$. If $R_{1234}\geq 0$, then we have
	$$ R_{1313}+R_{1414}+R_{2323}+R_{2424}\geq 0.$$
	On the other hand, if $R_{1234}<0$ then $R_{2134}>0$ and we can consider the four-frame $\{e_2,e_1,e_3,e_4\}$. Using (\ref{eq:WPIC}) again, we have
	$$ R_{2323}+R_{2424}+R_{1313}+R_{1414}-R_{2134}\geq 0,$$
	which implies that 
	$$ R_{1313}+R_{1414}+R_{2323}+R_{2424}\geq 0.$$
	Thus, we have
	$$ R_{ikik}+R_{ilil}+R_{jkjk}+R_{jljl} \geq 0 $$
	whenever $i,j,k,l \in \{1,...,n\}$ are mutually distinct.
	
	Next, by summing over $l\in \{1,...,n\}\backslash \{i,j,k\}$ in the above inequality for mutually distinct $i, j, k \in \{1,...,n\}$, we get
	\begin{equation*} 
		R_{ii}+R_{jj}-2R_{ijij}+(n-4)(R_{ikik}+R_{jkjk}) \geq 0.
	\end{equation*}
	If we further sum over $k\in \{1,...,n\}\backslash \{i,j\}$ for  mutually distinct $i, j \in \{1,...,n\}$, then it follows that 
	\begin{equation*}
		R_{ii}+R_{jj}-2R_{ijij} \geq 0.
	\end{equation*}
	Finally, summing over $j\in \{1,...,n\}\backslash \{i\}$ yields
	\begin{equation*}
		(n-4)R_{ii}+R \geq 0.
	\end{equation*}
	This finishes the proof of the lemma.
\end{proof}

% new subsection %%%%%%%%%%%%%%%%%%%%%%%%%%%%%%%%%%%%%%%%%%%%%%%%%%%%%%%%%%
\subsection{Curvature decomposition of four-manifolds}

For any oriented Riemannian 4-manifold $(M^4,g)$, we consider the decomposition of the bundle of 2-forms 
\begin{equation} \label{eq:decompof2forms}
\wedge^2(M) = \wedge^{+}(M) \oplus \wedge^{-}(M), 
\end{equation}
by {\em self-dual} 2-forms $\wedge^+ (M)$  and {\em anti-self-dual} 2-forms $\wedge^- (M)$, as well as the 
corresponding decomposition of the curvature operator, 
\begin{equation} \label{eq:CODecomposition}
	Rm = 
	\begin{pmatrix}
		A & B \\
		{}^{t}B & C
	\end{pmatrix}
	=
	\begin{pmatrix}
		W^+ + \frac{R}{12}I & \mathring{Rc} \\
		\mathring{Rc} & W^- +\frac{R}{12}I
	\end{pmatrix}.
\end{equation}
Here, $W^{\pm}$ denote the self-dual and anti-self-dual Weyl curvature tensors, respectively, $\mathring{Rc}$ denotes the traceless Ricci tensor\footnote{More precisely, the operator $B:\wedge^{-}(M)\to \wedge^{+}(M)$ is given by $\mathring{Rc}\KN g$, the Kulkarni-Nomizu product of $\mathring{Rc}$ and $g$. In particular, $B$ is identically zero when $(M^4, g)$ is Einstein.}, and $R$ denotes the scalar curvature.

Now, let us denote by $$A_1\leq A_2 \leq A_3 \qquad {\mbox{and}} \qquad  C_1 \leq C_2 \leq C_3$$ the eigenvalues of $A=W^+ + \frac{R}{12}I$ and $C=W^- + \frac{R}{12}I$, respectively,  and
$$a_1\leq a_2 \leq a_3 \qquad {\mbox{and}} \qquad  c_1 \leq c_2 \leq c_3$$ the eigenvalues of $W^+$ and $W^-$, respectively.  
Clearly,  we have
	$$  A_{1}=a_1+\frac{R}{12} , \quad  A_{2}=a_2+ \frac{R}{12}, \quad A_{3}=a_3+ \frac{R}{12},   $$
	$$  C_{1}=c_1+ \frac{R}{12}, \quad  C_{2}=c_2+ \frac{R}{12}, \quad C_{3}=c_3+\frac{R}{12}.  $$
Moreover, since $a_1+a_2+a_3=c_1+c_2+c_3=0$, we have 
 $$ \tr \ \!A = \tr  \ \!C = \frac{R}{4}.$$
Note also that $(M^n, g)$ is 
\begin {itemize} 

\smallskip
\item  {\it PIC} if and only if  $A_1+A_2 >0$ and $C_1+C_2>0$ \cite{Ha:97};

\smallskip
\item  {\it WPIC} if and only if $A_1+A_2 \geq 0$ and $C_1+C_2\geq 0$;

\smallskip
\item {\it Half PIC} if and only if either $A_1+A_2 >0$ or $C_1+C_2>0$;

\smallskip
\item  {\it Half WPIC}  if and only if either $A_1+A_2 \geq 0$ or $C_1+C_2\geq 0$.

\end {itemize}

\smallskip

Next, we let $0\leq B_1 \leq B_2 \leq B_3$ be the {\it singular eigenvalues} of $B$, 
$$ \lambda_1 \leq \lambda_2 \leq \lambda_3 \leq \lambda_4 \qquad {\mbox{and}} \qquad \Lambda_1 \leq \Lambda_2 \leq \Lambda_3 \leq \Lambda_4$$ be the eigenvalues of the traceless Ricci tensor $\mathring{Rc}$ and  the Ricci tensor $Rc$, respectively,   so that 
\begin{equation}\label{Rc-eigen}
 \Lambda_i=\lambda_i + \frac R 4, \qquad 1\le i\le 4.
\end{equation}

\begin{lem} \label{lem:RcandtracelessRc}
	Let $(M^4,g)$ be a 4-dimensional (oriented) Riemannian manifold.  Then, the sum of the least two eigenvalues of $Rc$ is given by
	\begin{equation*}
		\Lambda_1+\Lambda_2 = \frac 1 2(R-4B_3).
	\end{equation*}
\end{lem}

\begin{proof}
	It is well known (see, e.g.,  equation \cite[(2.7)]{Cao-Xie}) that 
	\begin{equation*} \label{traceless Rc&B}
		\begin{split}
			\mathring{Rc} = 
			\begin{pmatrix}
				B_{11}+B_{22}+B_{33} & B_{32}-B_{23} & B_{13}-B_{31} & B_{21}-B_{12} \\
				B_{32}-B_{23} & B_{11}-B_{22}-B_{33} & B_{21}+B_{12} & B_{13}+B_{31} \\
				B_{13}-B_{31} & B_{21}+B_{12} & B_{22}-B_{11}-B_{33} & B_{23}+B_{32} \\
				B_{21}-B_{12} & B_{13}+B_{31} & B_{23}+B_{32} & B_{33}-B_{11}-B_{22}
			\end{pmatrix},
		\end{split}
	\end{equation*}
	where $B_{ij}$ are the entries of the matrix $B$ in (\ref{eq:CODecomposition}). Now, we choose a frame $\{e_1,e_2,e_3,e_4\}$ such that the traceless Ricci tensor $\mathring{Rc}$ is diagonoal so that the matrix $B$ is also diagonal, with $B_{11}\leq B_{22}\leq B_{33}$. Then, we have
	\begin{equation} \label{eq:lambda}
		\begin{split}
			\lambda_{1}=B_{11}-B_{22}-B_{33} , \\
			\lambda_{2}= B_{22}- B_{11}-B_{33}  , \\
		\end{split}
		\quad \quad
		\begin{split}
			\lambda_{3}= B_{33}- B_{11}-B_{22}  , \\
			\lambda_{4}= B_{11}+B_{22}+B_{33}.
		\end{split}
	\end{equation}
In particular, 
\begin{equation} \label{eq:lambda sum}
\lambda_{1} + \lambda_{2}=-2B_{33}.
\end{equation}
		
On the other hand, by equation \cite[(2.5)]{Cao-Xie}, we have
	\begin{equation*}
		\begin{split}
			4B_{11}=\Lambda_{4}+\Lambda_{1}-\Lambda_{2}-\Lambda_{3}, \\
			4B_{22}=\Lambda_{4}+\Lambda_{2}-\Lambda_{3}-\Lambda_{1}, \\
			4B_{33}=\Lambda_{4}+\Lambda_{3}-\Lambda_{1}-\Lambda_{2}, \\
		\end{split}
	\end{equation*}
	which imply  $B_{33}\geq  B_{22}\geq 0$. 
Since $0\leq B_1\leq B_2\leq B_3$ are the singular eigenvalues of $B$, i.e., they are the eigenvalues of the matrix $\sqrt{B {}^tB}$, it follows that 
	\begin{equation} \label{eq:singularB}
		B_1= |B_{11}|=\pm B_{11},\quad B_2=|B_{22}|=B_{22},\quad B_3=|B_{33}|=B_{33}. 
	\end{equation}
Therefore, Lemma \ref{lem:RcandtracelessRc} follows immediately from (\ref{Rc-eigen}), (\ref{eq:lambda sum}), and (\ref{eq:singularB}). 
\end{proof}

% new subsection %%%%%%%%%%%%%%%%%%%%%%%%%%%%%%%%%%%%%%%%%%%%%%%%%%%%%%%%%%
\subsection{Some basic facts about Ricci solitons and ancient Ricci flows}

First, we recall the following basic identities satisfied by gradient Ricci solitons. 
\begin{lem} \textup{\bf (Hamilton \cite{Ha:93})} \label{lem:Hamilton}
	Let $(M^n,g,f)$ be an n-dimensional gradient Ricci soliton satisfying Eq. (\ref{eq:Riccisoliton}). Then
	\begin{itemize}
		\item[(i)] $ R + \Delta f = n\rho; $
		\smallskip
		\item[(ii)] $ \na_iR = 2R_{ij}\na_jf;$
		\smallskip
		\item[(iii)] $ R + |\na f|^2 = 2\rho f + C_0.$
	\end{itemize}
	where $C_0$ is some constant.
\end{lem}

Next, we recall that for any complete gradient Ricci soliton $(M^n,g,f)$ satisfying Eq. (\ref{eq:Riccisoliton}) there associates a canonical solution $g(t)$ to the Ricci flow defined by 
\begin{equation} \label{eq:canonicalRF}
	g(t)=(1-2\rho t)\Phi^{\ast}(t)(g) \qquad {\mbox{and}} \qquad g(0)=g, 
\end{equation}
where $\Phi(t)$ is the one-parameter family of diffeomorphisms generated by $\frac {\na f} {(1-2\rho t)}$, with $\Phi(0)=\textup{id}_M$.

We also recall the following curvature evolution equations in dimension $n=4$.
\begin{lem} \textup{\bf (Hamilton \cite{Ha:86})} \label{lem:RFequations}
	Let $(M^4,g(t))$ be a 4-dimensional complete solution to the Ricci flow. Then
	\begin{gather*}
		\pa_t R = \Delta R + 2|Rc|^2, \\
		\pa_t Rm =\Delta Rm + 2(Rm^2+Rm^{\sharp}), \\
		\pa_t A =\Delta A + 2(A^2 +2A^{\sharp} +B {}^tB), \\
		\pa_t B =\Delta B + 2(AB +BC + 2B^{\sharp}), \\
		\pa_t C =\Delta C + 2(C^2 +2C^{\sharp} +{}^tB B).
	\end{gather*}
Here, for any $3\times 3$ matrix $D$, $D^2$ denotes the matrix square of $D$, and $D^{\sharp}$ is the transpose of the adjoint matrix of $D$. 
\end{lem}

\begin{rmk}
	 Except for the first equation in Lemma \ref{lem:RFequations}, the factor $2$ in the rest of the equations differs from the corresponding equations in \cite{Ha:86} due to our slightly different definition of the inner product on $\wedge^2(M)$ as given in \cite[(2.4)]{Cao-Xie}.
\end{rmk}

Finally, we shall need several very useful properties concerning complete ancient solutions to the Ricci flow, of which gradient steady and shrinking Ricci solitons are special cases. 

\begin{lem} \textup{\bf (Chen \cite{ChenBL:09})} \label{lem:Chen}
	Let $g (t)$ be a complete ancient solution to the Ricci flow on $M^n$. Then, it has nonnegative scalar curvature $R\geq 0$  for each $t$.
\end{lem}

\begin{rmk} \label{rmk:R>0}
	In the special case of a gradient steady Ricci soliton $(M^n,g,f)$, by \cite[Proposition 3.2]{Petersen-W:09}, either the scalar curvature $R>0$ everywhere or $(M^n,g,f)$ is Ricci-flat.
\end{rmk}

\begin{rmk}
Not only the above result of Chen is very useful, but also Chen's main idea\footnote{Motivated in part by Hamilton's paper \cite[p.2]{Ha:93b} on eternal solutions to the Ricci flow.} of the proof is quite significant and has been widely used in the study of complete ancient solutions; see, e.g., \cite{Li-Ni:20, Cao-Xie, Cho-Li:23}.  Very recently, Cho and Li \cite{Cho-Li:23} have nicely summarized the essential feature of Chen's argument and formulated it as the following general lemma.
\end{rmk}
\begin{lem}  \textup{\bf (Chen's Lemma \cite[Corollary 2.4] {Cho-Li:23})} \label{Chen's lem}
	Let $(M^n,g(t))_{t\in (-\infty,0]}$ be an n-dimensional complete ancient solution to the Ricci flow. 
Suppose $u: M\times (-\infty,0] \rightarrow \R$ is a Lipschitz function and satisfies the differential inequality
\[(\pa_t -\Delta) u\geq \delta u^2\] 
in the barrier sense for some constant $\delta >0$. Then, $u \geq 0$ on $M \times (-\infty,0]$.
\end{lem}

\begin{lem} \textup{\bf (\cite[Proposition 3.1]{Cao-Xie})} \label{lem:Cao-Xie}
	Let $(M^4, g(t))$ be a 4-dimensional complete ancient solution to the Ricci flow.
	\begin{enumerate}

		\item[(a)] If $g(t)$ has half nonnegative isotropic curvature, then either $A\geq 0$ or $C\geq 0$.
	
		\item[(b)] If $g(t)$ has half positive isotropic curvature, then either $A > 0$ or $C > 0$.  
	\end{enumerate}

\end{lem}

\begin{rmk} 
	We point out that Lemma \ref{lem:Cao-Xie} is valid under some slightly weaker assumptions; see Proposition \ref{pro:positiveC1} for details. 
\end{rmk}

\medskip
% new section %%%%%%%%%%%%%%%%%%%%%%%%%%%%%%%%%%%%%%%%%%%%%%%%%%%%%%%%%%%%
\section{4D Gradient Steady Ricci solitons with (half) WPIC}

In this section, we study the geometry and classification of 4-dimensional complete gradient steady Ricci solitons, satisfying
\begin{equation}  \label{eq:steader}
	Rc + \na^2 f= 0,
\end{equation}
with WPIC or half WPIC. By Lemma \ref{lem:Chen} and scaling the metric $g$, if necessary, we can normalize the potential function $f$ so that 
\begin{equation}  \label{eq:steadernormalized_f}
	R+|\nabla f|^2=1. 
\end{equation}

First of all, we prove {\bf Theorem \ref{thm:4Dancient}} for 4-dimensional complete ancient solutions with WPIC which is restated as 

\begin{thm} \label{prop:4Dancient}
	Let $(M^4,g(t))$ be a 4-dimensional complete ancient solution to the Ricci flow.
	\begin{itemize}
		\item [(a)] If $g(t)$ has nonnegative isotropic curvature, then it has 2-nonnegative Ricci curvature. Moreover, it satisfies the curvature estimate $$|Rm| \leq R.$$
		
		\item [(b)] If $g(t)$ has positive isotropic curvature, then it has 2-positive Ricci curvature.
	\end{itemize}
\end{thm}

\begin{proof}
(a) We first prove that $g(t)$ has 2-nonnegative Ricci curvature. By Lemma \ref{lem:RcandtracelessRc}, the sum of the least two eigenvalues of $Rc$ is given by
\[ \Lambda_1+\Lambda_2= \frac 1 2 (R-4B_3).\]
Thus, 2-nonnegative Ricci curvature is equivalent to  $u:=R-4B_3 \geq0.$

Now, we consider the ordinary differential inequality satisfied by the Lipschitz function $u$ in the barrier sense. By Lemma \ref{lem:RFequations} and $|\mathring{Rc}|^2=4|B|^2$, we have
\begin{equation*}
	\begin{split}
		\frac{d}{dt}u &\geq 2|Rc|^2 - 8(A_3B_3+C_3B_3+2B_1B_2) \\
		&=\frac{1}{2}R^2+8|B|^2 - 8(A_3B_3+C_3B_3+2B_1B_2) \\
		&=\frac{1}{2}\left( R-4B_3\right)^2+ 4RB_3 +8(B_2-B_1)^2 - 8(A_3+C_3)B_3 \\
		&\geq \frac{1}{2}u^2+ 4(R-2A_3-2C_3)B_3.
	\end{split}
\end{equation*}
On the other hand, since $A_1+A_2+A_3=C_1+C_2+C_3=\frac{R}{4}$, we have 
\[ R-2A_3-2C_3=2(A_1+A_2+C_1+C_2). \]
Hence, we obtain
\begin{equation} \label{eq:ODEofu}
	\begin{split}
		\frac{d}{dt}u & \geq \frac{1}{2}u^2 + 8(A_1+A_2 +C_1+C_2)B_3 \\
               & \geq \frac{1}{2}u^2,
	\end{split}
\end{equation}
where, in the last inequality, we have used the assumption of WPIC, i.e., $A_1+A_2\geq 0$ and $C_1+C_2\geq 0$. Hence, by Lemma \ref{Chen's lem}, we have $R-4B_3\ge 0$, i.e.,  $Rc$ is 2-nonnegative.

Next, we prove the curvature estimate $|Rm|\leq  R$. By Lemma \ref{lem:Cao-Xie}, we have $A\ge 0$ and $C\ge 0$. Also, we know that $\tr A=\tr C=R/4.$ Hence, $|A|^2 \le \tr(A)^2 = \frac {1}{16}{R^2}$. Similarly, we have $|C|^2\leq \frac {1}{16}{R^2}$. Since the Ricci tensor of $(M^4, g, f)$ is 2-nonnegative, it is not hard to see that $|Rc|^2\leq R^2$. Indeed, let $\Lambda_1\le \Lambda_2\le \Lambda_3\le \Lambda_4$ be the eigenvalues of $Rc$, then, as $ \Lambda_1+\Lambda_2\geq 0$, we obtain
\begin{equation*}
	\begin{split}
		|Rc|^2 &  = \sum_{i=1}^4 \Lambda^2_i  \leq 2\Lambda_2^2+ \Lambda_3^2+ \Lambda_4^2 \leq (\Lambda_3+\Lambda_4)^2 \\
		& \leq  \left(\sum_{i=1}^4 \Lambda_i\right)^2= R^2.
	\end{split} 
\end{equation*}
On the other hand, 
\[|Rc|^2 = |\mathring{Rc}|^2 + \frac{1}{4}R^2 = 4|B|^2 +\frac{1}{4}R^2.\]
Therefore, $4|B|^2 +\frac{1}{4}R^2\leq R^2$ and 
\begin{equation} \label{eq:boundedRm}
	\begin{split}
		|Rm|^2 & \leq 2(|A|^2 + |C|^2 + |B|^2) \\
		& \leq \frac 1 4 R^2+ 2|B|^2 \leq \frac 5 8 R^2.
	\end{split} 
\end{equation}
This completes the proof of curvature bound.
	
\smallskip
(b) Suppose that $g(t)$ has PIC. We prove 2-positive Ricci curvature by contradiction. We consider the quadratic form $Z:=RI-4\sqrt{B {}^tB}$, where $I$ is the 3 by 3 identity matrix. By part (a), we know that $Z\geq 0$ and that  2-positive Ricci curvature is equivalent to $Z>0$. Now, we denote the eigenvalues of $Z$ by
\begin{equation}
	0\leq Z_1\leq Z_2\leq Z_3.
\end{equation}
Assume that $Z$ has a null eigenvector at some space-time point $(x_0, t_0)$. Then, by using a similar argument as in the proof of \cite[Proposition 3.1(b)]{Cao-Xie}, we get a contradiction. Alternatively, by a standard strong maximum principle argument (cf. \cite{Ha:93b}), it follows from (\ref{eq:ODEofu}) that if $Z_1=0$ attains its minimum at $(x_0, t_0)$, we must have
\begin{equation} \label{eq:equationofu}
	A_1+A_2+C_1+C_2 =0
\end{equation}
at $(x_0,t_0)$. 
On the other hand, by Lemma \ref{lem:Cao-Xie} (b), we know that $A> 0$ and $C> 0$. This is a contradiction. 
Hence, the Ricci tensor is 2-positive on $M^4\times (-\infty,0]$.

This completes the proof of Theorem \ref{prop:4Dancient}  (i.e., Theorem \ref{thm:4Dancient}).
\end{proof}

\begin{rmk}
	For $n\geq 5$, Li and Ni \cite{Li-Ni:20} proved that $n$-dimensional complete ancient solutions to the Ricci flow with WPIC must have 2-nonnegative Ricci curvature; see \cite[Proposition 5.2]{Li-Ni:20}. However, their proof does not extend to the $n=4$ case, since the key inequality \cite[Lemma 4.1]{Li-Ni:20} in the proof holds only for $n\geq 5$.
\end{rmk}

\medskip
{\bf Proof of Theorem \ref{thm:WPIC}.} Let $(M^4, g, f)$ be a $4$-dimensional complete noncompact, non-flat, gradient steady Ricci soliton with nonnegative isotropic curvature.

First of all, we claim that if the holonomy group $\text{Hol}^{0} (M^4, g)$ is $\text{SO}(4)$ then $(M^4, g, f)$ has PIC. Indeed, consider  the canonical Ricci flow $g(t)=\Phi(t)^{\ast}(g)$ induced by the steady soliton $(M^n,g,f)$, with $g(0)=g$, for $t\in [0,1]$. Here, $\Phi(t)$ is the one-parameter family of diffeomorphisms generated by the vector field $\nabla f$. By \cite[Proposition 8]{Brendle-S:08}, for any fixed time $t \in (0,1)$, the set of all orthonormal four-frames $\{e_1,e_2,e_3,e_4\}$ on $(M^4,g(t))$ satisfy
\begin{equation} \label{eq:WPIC=0}
	R_{1313}+R_{1414}+R_{2323}+R_{2424}-2R_{1234}=0
\end{equation}
is invariant under parallel transport with respect to the metric $g(t)$. Moreover,  as $g(t)=\Phi(t)^{\ast}(g)$ with $g(0)=g$, we conclude that all orthonormal four-frames $\{e_1,e_2,e_3,e_4\}$ on $(M^n,g,f)$ satisfying (\ref{eq:WPIC=0}) is also invariant under parallel translation with respect to $g$. Thus, by \cite[Corollary 9.14]{Brendle:10b}, if the holonomy group $\text{Hol}^{0} (M^4, g)$ is $\text{SO}(4)$ then $(M^4, g, f)$ has PIC.
	
Now, we are ready to conclude our proof. 

\smallskip
{\bf Case 1:} $(M^4,g,f)$ is locally reducible.  In this case, since $(M^4, g, f)$ is non-flat, its universal cover must be $N^3 \times \R$, or $\Sigma \times \Sigma$, or $\Sigma \times \R^2$, where $\Sigma^2$ is the cigar soliton and $N^3$ is either the Bryant soliton or a flying wing. 

\smallskip
{\bf Case 2:} $(M^4,g,f)$ is locally irreducible and locally symmetric. Then, as $(M^4,g,f)$ is a steady Ricci soliton it must be Ricci flat. Together with the third inequality in Lemma \ref{lem:WPIC},  or by \cite[Theorem 7.61]{Besse}, this would imply $(M^4,g,f)$ is flat.
So, Case 2 cannot occur when $(M^4,g,f)$ is assumed non-flat. 
\smallskip
	
{\bf Case 3:} $(M^4,g,f)$ is locally irreducible and not locally symmetric.  Then, by Berger's holonomy classification, either $\text{Hol}^{0} (M^4, g)=\text{SO}(4)$, or $\text{Hol}^{0} (M^4, g)=\text{U}(2)$, or $\text{Hol}^{0} (M^4, g)=\text{SU}(2)$. If $\text{Hol}^{0} (M^4, g)=\text{SO}(4)$ then, from the above, we know that $(M^4,g,f)$ must have PIC and 2-positive Ricci curvature. On the other hand, if $\text{Hol}^{0} (M^4, g)=\text{SU}(2)$ then $(M^4,g,f)$ is Calabi-Yau, hence flat due to WPIC. But this is ruled out by the non-flatness assumption.  Finally, if $\text{Hol}^{0} (M^4, g)=\text{U}(2)$, then $(M^4,g)$ is K\"ahler, hence $(M^4, g, f)$ is a gradient steady K\"ahler-Ricci soliton. On the other hand, any gradient steady K\"ahler-Ricci soliton of complex dimension two with WPIC necessarily has nonnegative curvature operator; see, e.g., \cite[Lemma 4.6]{Cho-Li:23}. Moreover, since $(M^4,g,f)$ is locally irreducible, not locally symmetric and has $\text{Hol}^{0} (M^4, g)=\text{U}(2)$, by using \cite[Theorem 8.3]{Ha:86} and a similar argument as in \cite{Cao-Chow:86}, it follows that $(M, g)$ has $Rm>0$ when restricted to real $(1, 1)$-forms $\wedge^{1,1}_{\R} (M)$. 
	This completes the proof of Theorem \ref{thm:WPIC}.
\hfill $\square$

\medskip
Finally, we restate and prove Theorem \ref{thm:halfWPIC}. 

\begin{thm} \label{prop:halfWPIC}
	Let $(M^4, g, f)$ be a 4-dimensional complete, non-flat, gradient steady Ricci soliton with half nonnegative isotropic curvature. Then, either
	\begin{itemize}
		\item[(i)] $(M^4, g, f)$ has half positive isotropic curvature, or
		\smallskip
		\item[(ii)] $(M^4, g, f)$ is Calabi-Yau, hence locally hyperK\"ahler, or
		\smallskip
		\item[(iii)] $(M^4, g, f)$ is Ricci-flat and self-dual, or                        
		\smallskip		
		\item[(iv)] $(M^4, g, f)$ is a locally irreducible, non-Ricci-flat, gradient steady K\"ahler-Ricci soliton, or
		\smallskip	
		\item[(v)] $(M^4, g, f)$ is isometric to a quotient of either $N^3\times \R$, or $\Sigma^2\times \Sigma^2$, or $\Sigma^2\times \R^2$. Here $\Sigma^2$ is the cigar soliton, and $N^3$ is either the Bryant soliton or a flying wing.
	\end{itemize}
\end{thm}

\begin{proof}
	
	First of all,  by  Lemma \ref{lem:Cao-Xie}, $(M^4,g,f)$ has either $A\geq 0$ or $C\geq 0$. Without loss of generality, we may assume $A\geq 0$.
	
	Next, we observe that $\ker (A)$ is invariant under parallel translation. Indeed, as in the proof of Theorem \ref{thm:WPIC}, we consider the canonical Ricci flow $g(t)=\Phi(t)^{\ast}(g)$ induced by the steady soliton $(M^n,g,f)$, with $g(0)=g$, for $t\in [0,1]$. Note that, by Lemma \ref{lem:Hamilton}, the evolution equation of the matrix $A(t)$ for $g(t)$ is given by
	\begin{equation*}
		\pa_t A = \Delta A + 2(A^2 +2A^{\sharp} +B{}^{t}B).
	\end{equation*}
	Since $A(t)\geq 0$, it follows that 
	\[ A(t)^2 +2A(t)^{\sharp} +B(t) {}^tB(t)\geq 0, \qquad  {\mbox{for all}}   \ t\in[0, 1].\] 
	Thus, by Hamilton's strong maximum principle (see, e.g., \cite[Theorem 2.2.1]{Cao-Zhu:06}), there exists an interval $0<t<\delta$ over which the rank of $A(t)$ is constant, and $\ker (A(t))$ is invariant under parallel translation and invariant in time. As $g(t)=\Phi(t)^{\ast}(g)$ with $g(0)=g$, we conclude that $\ker(A)=\ker(A(0))$ is also invariant under parallel translation for $(M^n,g,f)$. 
	
	\smallskip
	\noindent {\bf Claim:} If $(M^4, g, f)$ has half WPIC and the restricted holonomy group $\text{Hol}^{0} (M^4, g)$ is $\text{SO}(4)$, then either it is Ricci flat, or $A > 0$ hence $(M^4, g, f)$ has half PIC.

 	\smallskip
	
	We essentially follow the same argument as in the proof of \cite[Theorem 1.2]{Cao-Xie} and argue by contradiction. Suppose that $(M^4, g, f)$ is not Ricci flat and there exist a point $p\in M^4$ and a self-dual $\vphi_1 \in \wedge^+(M_p)$ such that
	$ A (\vphi_1,\vphi_1) = 0. $
	It is then clear that $\vphi_1$ is a null eigenvector corresponding to the smallest eigenvalue $A_1=0$ at $p$. Now, by \cite[Lemma 6.1]{Derdzinski:00}, the self-dual $\vphi_1 \in \wedge^+(M_p)$ can be expressed as 
	\begin{equation*} \label{eq:vphi1}
		\vphi_1 = \frac{1}{\sqrt{2}}\left( e_1\wedge e_2 + e_3\wedge e_4\right)
	\end{equation*}
	for some positively oriented orthonormal frame $\{e_1, e_2, e_3, e_4\}$.
	Meanwhile, suppose $\vphi_3\in \wedge^+(M_p)$ is an eigenvector corresponding to the largest eigenvalue $A_3$ at $p$. Then, using \cite[Lemma 6.1]{Derdzinski:00} again, we can find another positively oriented orthonormal frame $\{v_1, v_2, v_3, v_4\}$ such that
	\begin{equation*} \label{eq:vphi3}
		\vphi_3 = \frac{1}{\sqrt{2}}\left( v_1\wedge v_2 + v_3\wedge v_4\right). 
	\end{equation*}		
	Since $\text{Hol}^{0} (M^4, g)=\text{SO}(4)$, there exists a closed loop $\gamma$ based at $p$ such that 
	$$ v_i = P_{\gamma}e_i,\quad i=1, \cdots, 4,  $$
	where $P_{\gamma}$ denotes the parallel transport along $\gamma$. It then follows that 
$$ A_3=A(\vphi_3,\vphi_3) = A(\vphi_1,\vphi_1)=A_1=0,$$ since $\ker (A)$ is invariant under parallel translation. This would imply that the scalar curvature $R=4(A_1+A_2+A_3)=0$ at $p$. By Remark \ref{rmk:R>0}, $(M^4,g,f)$ must be Ricci flat, which is a contradiction. Thus, $A > 0$. Therefore,  $(M^4, g, f)$ has PIC and the Claim is proved.

	Now, we are ready to finish the proof of Theorem \ref{prop:halfWPIC}

\smallskip
{\bf Case 1:} $(M^4,g,f)$ is locally reducible.  In this case, since $(M^4, g, f)$ is non-flat, its universal cover must be $N^3 \times \R$, or $\Sigma \times \Sigma$, or $\Sigma \times \R^2$, where $\Sigma^2$ is the cigar soliton and $N^3$ is either the Bryant soliton or a flying wing. 

\smallskip
{\bf Case 2:} $(M^4,g,f)$ is locally irreducible and locally symmetric. Then, $(M^4,g,f)$ must be Ricci flat. However, Ricci-flat symmetric space is necessarily flat (cf. \cite[Theorem 7.61]{Besse}).  So Case 2 doesn't occur since $(M^4,g,f)$ is non-flat by assumption. 

\smallskip	
{\bf Case 3:} $(M^4,g,f)$ is locally irreducible and non-symmetric.  Then, by Berger's holonomy classification, either $\text{Hol}^{0} (M^4, g)=\text{SO}(4)$, or $\text{Hol}^{0} (M^4, g)=\text{U}(2)$, or $\text{Hol}^{0} (M^4, g)=\text{SU}(2)$. If $\text{Hol}^{0} (M^4, g)=\text{SO}(4)$ then, from the {\bf Claim} above, $(M^4,g,f)$ either is Ricci flat or has half PIC. If it is Ricci flat, then it implies that $A=W^{+}$ and $C=W^{-}$ because scalar curvature $R=0$. However, by Lemma \ref{lem:Cao-Xie}, we also have either  $A\geq 0$ or $C\geq 0$. Moreover, $\tr W^+=\tr W^-=0$. Hence, it follow that either $W^{+}=0$ or $W^-=0$, i.e., $(M^4,g,f)$ is either anti-self-dual (ASD) or self-dual (SD). Moreover, it is well-known that Ricci-flat and ASD imply that the universal cover is hyperK\"ahler. 
On the other hand, if $\text{Hol}^{0} (M^4, g)=\text{U}(2)$, then $(M^4,g, f)$ is a gradient steady K\"ahler-Ricci soliton. Finally, if $\text{Hol}^{0} (M^4, g)=\text{SU}(2)$ then $(M^4,g,f)$ is Calabi-Yau, and its universal cover is hyperK\"ahler.

This finishes the proof of Theorem \ref{prop:halfWPIC} (Theorem \ref{thm:halfWPIC}).
\end{proof}

\begin{rmk} As mentioned in the introduction, Theorem \ref{thm:WPIC} and Theorem \ref{thm:halfWPIC} are valid under some slightly weaker assumption than WPIC (or half WPIC); see Remark  \ref{rmk:improving}.  
\end{rmk}

%\bigskip
% new section %%%%%%%%%%%%%%%%%%%%%%%%%%%%%%%%%%%%%%%%%%%%%%%%%%%%%%%%%%%%
\section {4D gradient Expanding Ricci solitons with WPIC}
In this section, we study 4-dimensional complete gradient expanding Ricci solitons with WPIC and prove Theorem \ref{thm:WPICexpander}. 

By scaling the metric $g$ in Eq.(\ref{eq:Riccisoliton}), we may assume $\rho=-\frac{1}{2}$ so that the expanding soliton equation reduces to 
\begin{equation}  \label{eq:expander}
	Rc+ \na^2 f= -\frac{1}{2}g.
\end{equation}
By replacing $f$ by $f-C_0$ in Lemma \ref{lem:Hamilton}, we can also normalize the potential function $f$ so that
\begin{equation}  \label{eq:expandernormalized_f}
	R+|\nabla f|^2=-f.
\end{equation}

For the reader's convenience, we restate Theorem \ref{thm:WPICexpander} here with more explicit part (iii).

\begin{thm} \label{thm:expander}
	Let $(M^4, g, f)$ be a 4-dimensional complete noncompact, non-flat, gradient expanding Ricci soliton with nonnegative isotropic curvature. Then, either
	\begin{itemize}
		\item[(i)] $(M^4, g, f)$ has positive isotropic curvature, or
	\smallskip	
		\item[(ii)] $(M^4, g, f)$ is a locally irreducible gradient expanding K\"ahler-Ricci soliton (with WPIC), or
	\smallskip	
		\item[(iii)] $(M^4, g, f)$ is isometric to a quotient of either $N^3\times \R$, where $N^3$ is a 3-dimensional expanding Ricci soliton with $Rc> 0$, or $(\R^2,g_6(\nu)) \times \R^2$, or $(\R^2,g_6(\nu)) \times (\R^2,g_6(\nu))$, or $(\R^2,g_6(\nu)) \times (\R^2,g_7(\nu))$, or $(\R^2,g_6(\nu)) \times (\R_{\ast}^2,g_8(\nu))$, where $(\R^2,g_6(\nu))$, $(\R^2,g_7(\nu))$, and $(\R_{\ast}^2,g_8(\nu))$ are the 2D complete gradient expanding Ricci solitions in \cite[Theorem 1]{Bernstein-M:15}; see also \cite{Ramos:18}.
	\end{itemize}
\end{thm}

\begin{proof} 
	Let $(M^4, g, f)$ be a 4-dimensional complete noncompact, non-flat, gradient expanding Ricci soliton with nonnegative isotropic curvature (WPIC).
	
	First of all, as is well-known, the assumption of WPIC implies that the scalar curvature $R\ge 0$. 

Next, we claim that if the holonomy group $\text{Hol}^{0} (M^4, g)$ is $\text{SO}(4)$, then $(M^4, g, f)$ has PIC. Indeed, consider the canonical Ricci flow $g(t)=(1+t)\Phi(t)^{\ast}(g)$ induced by the expanding soliton $(M^n,g,f)$ with $g(0)=g$ for $t\in [0,1]$. As in the proof of Theorem {\ref{thm:WPIC}}, by \cite[Proposition 8]{Brendle-S:08}, for any fixed time $t \in (0,1)$, the set of all orthonormal four-frames $\{e_1,e_2,e_3,e_4\}$ on $(M^4,g(t))$ satisfy
	\begin{equation} \label{eq:expanderWPIC=0}
		R_{1313}+R_{1414}+R_{2323}+R_{2424}-2R_{1234}=0
	\end{equation}
	is invariant under parallel transport with respect to the metric $g(t)$, which implies that all orthonormal four-frames $\{e_1,e_2,e_3,e_4\}$ on $(M^n,g,f)$ satisfying (\ref{eq:expanderWPIC=0}) is also invariant under parallel translation with respect to $g$. Thus, by \cite[Corollary 9.14]{Brendle:10b}, if the holonomy group $\text{Hol}^{0} (M^4, g)$ is $\text{SO}(4)$ then $(M^4, g, f)$ has PIC.
	
\smallskip
	 Now, we follow essentially the same argument as in the proof of Theorem \ref{thm:WPIC}. 
	
\smallskip
	{\bf Case 1:} $(M^4,g,f)$ is locally reducible.  In this case, since $(M^4, g, f)$ is non-flat, it must be a finite quotient of either $N^3\times \R$, where $N^3$ is a 3-dimensional expanding Ricci soliton, or $(\R^2,g_6(\nu)) \times \R^2$, or $(\R^2,g_6(\nu)) \times (\R^2,g_6(\nu))$, or $(\R^2,g_6(\nu)) \times (\R^2,g_7(\nu))$, or $(\R^2,g_6(\nu)) \times (\R_{\ast}^2,g_8(\nu))$, where $(\R^2,g_6(\nu))$, $(\R^2,g_7(\nu))$, and $(\R_{\ast}^2,g_8(\nu))$ are the 2-dimensional complete gradient expanding Ricci solitions in \cite[Theorem 1]{Bernstein-M:15}\footnote{We point out that the expanding Ricci solitons $(\R^2,g_7(\nu))$ and $(\R_{\ast}^2,g_8(\nu))$ in \cite[Theorem 1]{Bernstein-M:15} are negatively curved. Therefore, the product of these two solitons and  the product of one of these two solitons with $\R^2$ all have negative scalar curvature, hence they do not appear in the list.}. On the other hand, since $N^3\times \R$ has WPIC, $N^3$ has nonnegative Ricci curvature $Rc\geq 0$. Moreover, by Hamilton's strong maximum principle (see, e.g., \cite[Theorem 2.2.1]{Cao-Zhu:06}), $N^3$ either has $Rc> 0$ or it splits. 
	
	{\bf Case 2:} $(M^4,g,f)$ is locally irreducible and symmetric. If so, $(M^4,g)$ would necessarily be  Einstein with negative scalar curvature. However, by assumption, $(M^4 ,g, f)$ has nonnegative scalar curvature $R\geq 0$ due to WPIC. Thus, Case 2 doesn't occur. 
	
	{\bf Case 3:} $(M^4,g,f)$ is locally irreducible and non-symmetric.  Then, by Berger's holonomy classification, either $\text{Hol}^{0} (M^4, g)=\text{SO}(4)$, or $\text{Hol}^{0} (M^4, g)=\text{U}(2)$, or $\text{Hol}^{0} (M^4, g)=\text{SU}(2)$. If $\text{Hol}^{0} (M^4, g)=\text{SO}(4)$ then, from the above, we know that $(M^4,g,f)$ must have PIC. On the other hand, if $\text{Hol}^{0} (M^4, g)=\text{U}(2)$, then $(M^4,g, f)$ is a gradient expanding K\"ahler-Ricci soliton (with WPIC). Finally, if $\text{Hol}^{0} (M^4, g)=\text{SU}(2)$, then $(M^4,g,f)$ is Calabi-Yau, in particular Ricci-flat. Again, it follows from the third inequality in Lemma \ref{lem:WPIC} that $(M^4,g,f)$ must be flat; but this is ruled out by the non-flatness assumption.

	This finishes the proof of Theorem \ref{thm:expander} (Theorem \ref{thm:WPICexpander}).
\end{proof}

\bigskip
% new section %%%%%%%%%%%%%%%%%%%%%%%%%%%%%%%%%%%%%%%%%%%%%%%%%%%%%%%%%%%%
\section{The improved results for 4D shrinking Ricci solitons}

In our earlier paper \cite{Cao-Xie}, we investigated 4-dimensional gradient shrinking Ricci solitons with half PIC or half WPIC. In this section, motivated by the recent work of Li and Zhang \cite{Li-Zhang:22}, we observe the following results, improving \cite[Theorem 1.1--Theorem 1.3]{Cao-Xie} and \cite[Corollary 1.1 \& Corollary 1.2]{Cao-Xie}.

\begin{thm} \label{thm:quadraticofC1}
	Let $(M^4, g, f)$ be an orientable 4-dimensional complete gradient shrinking Ricci soliton. 
	
	\smallskip
	\begin{enumerate}
		\item[(a)] If $(M^4, g, f)$ has $A_2\geq 0$ or $C_2\geq 0$, then $A\geq 0$ or $C\geq 0$.
		
		\smallskip
		\item[(b)] If $(M^4, g, f)$ has $A_2>0$ or $C_2>0$, then $A > 0$ or $C > 0$. Moreover, if $M^4$ is noncompact then there exists some constant $K>0$ such that the smallest eigenvalue $A_1$ of $A$, or $C_1$ of $C$, satisfies the estimate 
		$$ A_1 \geq \frac{K}{f}, \quad {\mbox{or}} \quad  \ C_1 \geq \frac{K}{f}.$$
	\end{enumerate}
\end{thm}

The proof of Theorem \ref{thm:quadraticofC1} follows from \cite[Proposition 3.2]{Cao-Xie}
and the following proposition, which generalizes \cite[Proposition 3.1]{Cao-Xie}.

\begin{prop} \label{pro:positiveC1}
	Let $(M^4, g (t))$ be an orientable 4-dimensional complete ancient solution to the Ricci flow.  
	
	\smallskip
	\begin{enumerate}
		\item[(a)] If $g(t)$ has $A_2\geq 0$ or $C_2\geq 0$ on $M^4$, then $A\ge 0$ or $C\geq 0$ on $M^4$.
		
		\smallskip
		\item[(b)] If $g(t)$ has $A_2>0$ or $C_2>0$ on $M^4$, then $A > 0$ or $C > 0$ on $M^4$.  
		
	\end{enumerate}
	In particular, (a) and (b) hold for complete gradient shrinking or steady Ricci solitons (including Einstein 4-manifolds with positive or zero scalar curvature as special cases). 
\end{prop}

\begin{rmk} The proof of \cite[Proposition 3.1(a)]{Cao-Xie} depends only on $A_2\ge 0$ or $C_2\ge 0$. Similarly,  the proof of \cite[Proposition 3.1(b)]{Cao-Xie} depends only on $A_2>0$ or $C_2>0$. Thus, the proof of Proposition \ref{pro:positiveC1} is identical to that of \cite[Proposition 3.1]{Cao-Xie}.
\end{rmk}

As a consequence of Theorem \ref{thm:quadraticofC1}, 
we obtain the following result in the gradient shrinking K\"ahler-Ricci soliton case because the proof of \cite[Corollary 1.1]{Cao-Xie} only relies on the fact that the shrinking K\"ahler-Ricci soliton has $C> 0$ or $C\ge 0$. 

\begin{cor} \label{thm:kahler}
	Let $(M^4,g,f)$ be a (canonically oriented) complete gradient shrinking K\"ahler-Ricci soliton of complex dimension two.
	
	\begin{enumerate}
		\item[(i)] If $(M^4,g,f)$ has $C_2>0$ on $M^4$, then it is, up to automorphisms, the complex projective space $\CP^2$. 
		
		\smallskip	
		\item[(ii)] If $(M^4,g,f)$ has $C_2\geq 0$ on $M^4$, then it is, up to automorphisms, one of the following: the complex projective space $\CP^2$, the product $\CP^1 \times \CP^1$, the cylinder $\CP^1 \times \C$, or the Gaussian soliton on $\C^2$.
		
	\end{enumerate}
\end{cor}

\begin{rmk}
	In Corollary \ref{thm:kahler}, $M^4$ is canonically oriented by the complex structure so the matrix $A$ has $A_1=A_2=0$ on $M^4$;  see, e.g.,  (2.14) in \cite{Cao-Xie}.
\end{rmk}

\medskip
By Theorem \ref{thm:quadraticofC1} and the fact that the strong maximum principle argument in \cite[Theorem 1.2]{Cao-Xie} depends only on $A> 0$ or $A\ge 0$,  we have

\begin{thm} \label{thm:dichotomy}
	Let $(M^4, g, f)$ be an oriented 4-dimensional complete gradient shrinking Ricci soliton with $A_2\geq 0$ on $M^4$. Then, $(M^4, g, f)$ either has $A>0$ on $M^4$, or is isometric to the Gaussian soliton $\R^4$ or a $\Z_2$ quotient of   $\rS^2 \times \rS^2$ or $\rS^2 \times \R^2$, or is a gradient shrinking K\"ahler-Ricci soliton. 
\end{thm}

\begin{rmk}\label{rmk:improving}
	Similarly, by Proposition \ref{pro:positiveC1} and the fact that the strong maximum principle argument in the proof of Theorem \ref{thm:WPIC} depends only on $A\geq 0$ and $C \geq 0$, we can improve Theorem \ref{thm:WPIC} so that the same result holds under the weaker assumption of $A_2\geq 0$ and $C_2\geq 0$. Similarly, Theorem \ref{thm:halfWPIC} is valid under the weaker assumption of $A_2\geq 0$. 
\end{rmk}

Consequently, as in \cite{Cao-Xie}, we obtain
\begin{cor} \label{cor:classification}
	Let $(M^4, g, f)$ be a 4-dimensional complete gradient shrinking Ricci soliton with $A_2\geq 0$ on $M^4$. Then, either
	\begin{enumerate}
		\item[(i)] $(M^4, g, f)$ has $A > 0$ on $M^4$, or 
		
		\smallskip
		\item[(ii)] $(M^4, g, f)$ is isometric to the Gaussian soliton $\R^4$ or a $\Z_2$ quotient of  $\rS^2 \times \rS^2$ or $\rS^2 \times \R^2$, or
		
		\smallskip
		\item[(iii)] $(M^4, g, f)$ is K\"ahler and, up to automorphisms, one of the following: a closed del Pezzo surface with its unique K\"ahler-Einstein or K\"ahler-Ricci soliton metric, the FIK \cite{FIK:03} soliton on the blowup of $\C^2$ at the origin,  or the BCCD  \cite {BCCD:22} soliton on the blowup of $\CP^1 \times \C$ at one point. 
		
	\end{enumerate}
\end{cor}

As a special case of Corollary \ref{cor:classification}, we have the following characterization of complete gradient {\it K\"ahler-Ricci shrinkers} in complex dimension two.  

\begin{cor} \textup{\bf (Theorem \ref{thm: compact}.)} \label{cor: compact}
	Let $(M^4, g, f)$ be an oriented 4-dimensional complete gradient shrinking Ricci soliton with $H^1 (M, \Z_2)=0$. 
	Then, $(M^4, g, f )$ is a K\"ahler-Ricci shrinker if and only if $A_2\ge 0$ (but $A_2\not >0$) on $M^4$.
\end{cor}

\begin{rmk}  
	For oriented {\bf positive Einstein} $4$-manifold $(M^4, g)$, recently Li and Zhang \cite{Li-Zhang:22} showed that if $H^1 (M, \Z_2)=0$ and $(M^4, g)$ is not isometric to a round 4-sphere $\rS^4$, then $(M^4, g)$ is K\"ahler-Einstein if and only if $A_2\ge 0$ on $M^4$; furthermore, if $A_2>0$ and $M^4$ is simply connected then  $(M^4, g)$ is isometric to $\rS^4$ with the round metric.  
	The proof in \cite{Li-Zhang:22} is based on the Weitzenb\"ock formula from Micallef-Wang \cite{MW:93} for $W^+$ with $\delta W^+=0$. Note that the combination of our Proposition \ref{pro:positiveC1}(b) and \cite[Theorem 4.2]{MW:93} also gives an alternative proof. 
\end{rmk}

\medskip
Finally, by Proposition \ref{pro:positiveC1}(b), \cite[Theorem 1.3]{Cao-Xie} is also valid under the weaker assumption of $A_2>0$ or $C_2>0$. 

\begin{thm} \label{thm:3}
	Let $(M^4,g,f)$ be an orientable 4-dimensional complete gradient Ricci shrinker such that its Ricci tensor has an eigenvalue with multiplicity 3. If $(M^4,g,f)$  has $A_2>0$ or $C_2>0$, then it is either isometric to a round $\rS^4$, or $\CP^2$ with the Fubini-Study metric, or a finite quotient of round cylinder $\rS^3 \times \R$.
\end{thm}

\smallskip
% new section %%%%%%%%%%%%%%%%%%%%%%%%%%%%%%%%%%%%%%%%%%%%%%%%%%%%%%%%%%%%
\section{Further Discussions}

\subsection {2-nonnegative Ricci curvature and 4D Ricci flow with WPIC}
By using the WPIC condition (i.e., $A$ and $C$ are weakly 2-positive), the same computation as in the proof of our Theorem \ref{prop:4Dancient}(a) and Hamilton's maximum principle \cite[Theorem 4.2]{Ha:86} imply that the {\it 2-nonnegative} Ricci curvature property is preserved by 4D Ricci flow with WPIC.

\begin{prop} \label{prop:2Rc_preserved}
	Let $(M^4,g(t))$, $t\in [0, T)$, be a complete solution to the Ricci flow with nonnegative isotropic curvature and bounded curvature. 
If $Rc$ is 2-nonnegative at $t=0$ then $Rc$ is 2-nonnegative on $0\le t < T$.
\end{prop}

\begin{proof}
By Lemma \ref{lem:RcandtracelessRc}, we know that 2-nonnegative Ricci curvature is equivalent to  \[R-4B_3 \geq0.\]
Moreover, as we have seen in the proof of Theorem \ref{prop:4Dancient}(a), 
\begin{equation*}
	\begin{split}
		\frac{d}{dt}\left( R-4B_3\right)  & \geq \frac{1}{2}\left( R-4B_3\right)^2 
+ 8(A_1+A_2 +C_1+C_2)B_3 \\
		&\geq \frac{1}{2}\left( R-4B_3\right)^2 \\
		&\geq 0,
	\end{split}
\end{equation*}
where, in the second inequality, we have used the assumption of WPIC, i.e., $A_1+A_2\geq 0$ and $C_1+C_2\geq 0$. Therefore, by Hamilton's maximum principle \cite[Theorem 4.2]{Ha:86}, the 2-nonnegative Ricci curvature condition is preserved under the 4-dimensional Ricci flow with WPIC. This finishes the proof of Proposition \ref{prop:2Rc_preserved}.
\end{proof}

\begin{rmk}
As we mentioned before, for $n\geq 5$, Li and Ni \cite{Li-Ni:20} showed that complete ancient solutions to the Ricci flow with WPIC must have {\bf  2-nonnegative} Ricci curvature; see \cite[Proposition 5.2]{Li-Ni:20}. By using essentially the same argument as in their proof, it follows that the 2-nonnegativity of $Rc$ is also preserved by $n$-dimensional $(n\geq 5$) Ricci flow with WPIC. However, their proof does not seem to extend to the $n=4$ case, since the key inequality in \cite[Lemma 4.1]{Li-Ni:20} holds only for $n\geq 5$. 
\end{rmk}

Next, we observe that the condition of 2-nonnegative Ricci curvature is preserved by the Ricci flow in dimension three. 
In this case, we can diagonalize the curvature operator $Rm$ with eigenvalues $m_1\leq m_2\leq m_3$. Then the Ricci tensor $Rc$ is diagonalized with eigenvalues $$\frac{1}{2}(m_1+m_2)\leq \frac{1}{2}(m_1+m_3)\leq \frac{1}{2}(m_2+m_3),$$ and the scalar curvature is given by $R=m_1+m_2+m_3$. Moreover, we have the following ODE system corresponding to the curvature evolution PDE (\cite{Ha:86}): 
\begin{equation} \label{eq:3dODE}
	\frac{d}{dt}m_1=m_1^2+m_2m_3,\quad \frac{d}{dt}m_2=m_2^2+m_1m_3,\quad \frac{d}{dt}m_3=m_3^2+m_1m_2.
\end{equation}

\begin{prop} \label{prop:3-D}
	Let $(M^3,g(t))$, $t\in [0, T)$, be a complete solution to the Ricci flow with bounded curvature. If $Rc$ is 2-nonnegative at $t=0$ then $Rc$ is 2-nonnegative on $0\le t < T$.
\end{prop}

\begin{proof}
	First of all, we note that 2-nonnegative Ricci curvature is equivalent to $2m_1+m_2+m_3 \geq 0$.
	
	Now, by (\ref{eq:3dODE}), when $2m_1+m_2+m_3 =0$, i.e., $m_3=-2m_1-m_2$, we have
	\begin{equation*}
		\begin{split}
			\frac{d}{dt}(2m_1+m_2+m_3) &= 2m_1^2 +2m_2m_3+m_2^2+m_1m_3+m_3^2+m_1m_2 \\
			&=2m_1^2 +2m_2(-2m_1-m_2)+m_2^2 \\
			&\quad +m_1(-2m_1-m_2)+(-2m_1-m_2)^2+m_1m_2 \\
			&=4m_1^2 \geq 0.
		\end{split}
	\end{equation*}	
	Therefore, by Hamilton's maximum principle \cite[Theorem 4.2]{Ha:86}, the 2-nonnegative Ricci curvature condition is preserved by the 3-dimensional Ricci flow. 
\end{proof}

\subsection{PIC vs. Uniformly PIC for 4D ancient solutions}

Under the stronger assumption of {\it uniformly} PIC, Cho-Li \cite{Cho-Li:23}  proved that any $4$-dimensional complete ancient solution $g(t)$ must have nonnegative curvature operator $Rm\ge 0$. Here, by uniformly PIC it means that $g(t)$ has PIC and satisfies the additional pointwise pinching condition
%\medskip
\begin{equation} \label{condition}
\max\{A_3, B_3, C_3\} \le \Lambda \min \{A_1+A_2, C_1+C_2\}
\end{equation}
for some constant $\Lambda>0$. Furthermore, they showed that such ancient solutions necessarily have bounded curvature, provided they are $\kappa$-noncollapsed. Combining these properties with the result \cite[Corollary 1.6]{BN:23} of Brendle and Naff,  it follows that any $4$-dimensional $\kappa$-noncollapsed, noncompact, complete ancient solution to the Ricci flow with uniformly PIC is isometric to either a family of shrinking cylinders (or their quotients) or the Bryant soliton. 

\begin{rmk}
However, each of the one-parameter family of $\Z_2\times O(3)$-invariant, yet non-rotationally symmetric,  4D steady gradient solitons constructed by Lai \cite{Lai1} has positive curvature operator $Rm>0$, hence PIC,  but is not uniformly PIC. Thus, Lai's examples show that the assumption of uniformly PIC is strictly stronger than that of PIC. In fact, if a complete steady gradient soliton $( M^4, g, f)$ has uniformly PIC then it is asymptotically cylindrical and has linear scalar curvature decay from above and below; see \cite{Brendle:14} for details. 
\end{rmk}

On the other hand, motivated by the result of B.-L. Chen \cite{ChenBL:09} that any $3$-dimensional
 complete ancient solution to the Ricci flow must have $Rm\geq 0$, it is also natural to ask the following

\medskip
\noindent {\bf Question 6.1\footnote{This question was raised by the first author at a colloquium talk given at Sun Yat-sen University (Guangzhou, China) in December, 2018.}.} Let $g(t)$ be a $4$-dimensional complete ancient solution to the Ricci flow on $M^4$ with PIC (or WPIC). Is it true that $g(t)$ necessarily has nonnegative curvature operator $Rm\ge 0$? 

\smallskip
We conclude our paper with the  following observation based on Theorem \ref{thm:4Dancient}.  

\begin{prop} \label{PIC vs. unif PIC}
Let $g(t)$, $-\infty <t\le 0$, be a $4$-dimensional complete ancient solution to the Ricci flow on $M^4$ with PIC. Suppose that 
\begin{equation} \label{weak unif}
A_3 \leq L (A_1+A_2) \qquad {\mbox{and}} \qquad C_3 \leq L (C_1+C_2), 
\end{equation}
or equivalently
\begin{equation} \label{weak unif}
R \leq 4(L+1) \ \!\max \{ A_1+A_2, C_1+C_2\}
\end{equation}
for some constant $L>0$. Then, $g(t)$ has uniformly PIC. 
\end{prop}

\bigskip
\noindent {\bf Acknowledgements.} We are grateful to Professor Richard Hamilton for helpful suggestions. We would also like to thank Dr. Pak-Yeung Chan, Dr. Florian Johne, Dr. Yi Lai,  and the anonymous referee for helpful comments. The first author's research was partially supported by a grant from the Simons Foundation.

\bigskip
% references %%%%%%%%%%%%%%%%%%%%%%%%%%%%%%%%%%%%%%%%%%%%%%%%%%%%%%%%%%%%

\end{document}